\documentclass{article}

\usepackage{hyperref}
\usepackage{amsthm,amssymb,amsmath}

\title{Using the subspace theorem to bound unit distances}
\author{Ryan Schwartz}

\newtheorem{theorem}{Theorem} 
\newtheorem{lemma}[theorem]{Lemma}     

\begin{document}
\maketitle

\begin{abstract}
We prove a special case of Erd\H{o}s' unit distance problem using
a corollary of the subspace theorem bounding the number of solutions of linear 
equations from a multiplicative group.  We restrict our attention to unit 
distances coming from a multiplicative group of rank $r$ not too large.  
Specifically, given $\varepsilon>0$ and $n$ points in the plane, we construct 
the unit distance graph from these points and distances and use the corollary 
above to bound certain paths of length $k$ in the graph giving at most 
$n^{1+\varepsilon}$ unit distances from the group above.  We require that the 
rank $r\le c\log n$ for some $c>0$ depending on $\varepsilon$.  This extends a 
result of J\'ozsef Solymosi, Frank de Zeeuw and the author where we only 
considered unit distances that are roots of unity.  Lastly we show that the 
lower bound configuration for the unit distance problem of Erd\H{o}s consists of 
unit distances from a multiplicative subgroup of the form above.

\end{abstract}

\section{Introduction}
In 1946 Erd\H os asked for the maximal possible number, $u(n)$, of unit 
distances among $n$ points in the plane.  He gave the lower bound $u(n)>
n^{1+c/\log\log n}$, using a $\sqrt{n}\times \sqrt{n}$ grid and conjectured that 
this was the true magnitude \cite{Erdo46}.  The best known upper bound is 
$u(n)<cn^{4/3}$, first proved by Spencer, Szemer\'edi and Trotter in 1984 
\cite{Spen84}.  This bound has many proofs, the simplest of which was the
proof by Sz\'ekely \cite{Szek97}, using the crossing inequality for graphs.  As 
a general reference for work done on the unit distance problem see 
\cite{Bras06}.

In \cite{Schw12a} J\'ozsef Solymosi, Frank de Zeeuw and the author showed that 
the number of unit distances between $n$ points in the plane with angle to the 
$x$-axis a rational multiple of $\pi$ is at most $n^{1+c/\sqrt{\log n}}$.  Such 
unit distances correspond to roots of unity.  The bound was found by using a 
result of H.B.~Mann bounding sums of roots of unity \cite{Mann65}.  This work is 
an extension of that result to unit distances from a group of finite 
rank---roots of unity correspond to rank 0.  The proof follows in almost the 
same way except we use the subspace theorem instead of Mann's result.

Our main tool is a corollary of the subspace theorem.  The subspace theorem was 
first proved by W.M.~Schmidt in 1972 \cite{Schm72}.  This theorem essentially 
says that solutions of linear equations, in a multiplicative subgroup of a field 
of finite rank, come from a finite number of linear subspaces.  A number of 
improvements, including quantitative versions, of this result have been made.  
The corollary we use bounds the number of such solutions depending on the rank 
and dimension.  This corollary was due originally to Evertse, Schlickewei and 
Schmidt \cite{Ever02}.  We will use an improvement of this result of Amoroso and 
Viada \cite{Amor09}.  The progression of these results is given in 
\cite{Ever02,Ever99}.  We will use the bound from \cite{Amor09} which is, to our 
knowledge, the best known bound for the corollary.  The subspace theorem is a 
very powerful result with a wealth of applications in number theory.  For some 
examples see \cite{Bilu07}. 

Consider two points $p,q\in\mathbb{R}^2$ with unit distance.  Considering the 
vector between these two points we get the complex number 
$z=z(p,q)=\overrightarrow{pq}$ with $|z|=1$.  We will restrict our attention to 
unit distances with $z$ coming from a multiplicative subgroup of $\mathbb{C}^*$ 
(the multiplicative group of nonzero complex numbers) of finite rank.  A 
subgroup $\Gamma\subset\mathbb{C}^*$ has rank $r$ if there exists a finitely 
generated subgroup $\Gamma_0\subset\Gamma$ with $r$ generators such that for 
every $x\in\Gamma$ there exists an integer $k\ge 0$ such that $x^k\in\Gamma_0$.

Suppose $\Gamma$ is a subgroup of $\mathbb{C}^*$ of finite rank $r$ and 
$a_1,a_2,\dots,a_k\in\mathbb{C}^*$.  A solution of the equation 
$a_1z_1+a_2z_2+\dots+a_kz_k=1$ is called \emph{nondegenerate} if no subsum of 
the left hand side vanishes.  That is $\sum_{j\in J} a_jz_j\ne 0$ for every 
nontrivial $J\subset\{1,2,\dots, k\}$.  We will consider the number $A(k,r)$ of 
nondegenerate solutions of this equation with $z_i\in\Gamma$.  We now give the 
corollary of the subspace theorem that we need.
\begin{theorem}
   \label{thm:subspace}
   Suppose $a_1, a_2, \dots, a_k\in\mathbb{C}^*$ and $\Gamma$ has finite rank 
   $r$.  Then the number of nondegenerate solutions of the equation 
   \begin{equation}
      a_1z_1+a_2z_2+\dots+a_kz_k=1 \label{eq:sum}
   \end{equation}
   with $z_i\in\Gamma$ is at most \[A(k,r)\le (8k)^{4k^4(k+kr+1)}.\]
\end{theorem}
Theorem~\ref{thm:subspace} was proved over an arbitrary algebraically closed 
field $K$ of characteristic $0$ but we only require it over $\mathbb{C}$.

We will use this to prove the following result.
\begin{theorem}
   \label{thm:main}
   Let $\varepsilon>0$.  Then there exist $n_0=n_0(\varepsilon)$ a positive 
   integer and $c=c(\varepsilon)>0$ such that given $n>n_0$ points in the plane, 
   the number of unit distances with $z$ coming from a subgroup $\Gamma \subset 
   \mathbb{C}^*$ with rank $r<c\log n$ is at most $ n^{1+\varepsilon}$.
\end{theorem}

We will prove this theorem in the next section and in Section~\ref{sec:lower} we 
will show that the lower bound of Erd\H{o}s satisfies the hypotheses of this 
theorem.

\section{Proof of the main result}
The proof of Theorem~\ref{thm:main} is quite similar to the proof of Theorem~2.2 
in \cite{Schw12a}.  The main difference is that we use the subspace theorem 
instead of Mann's result to get an upper bound for paths in the unit distance 
graph.

Suppose $G=G(V,E)$ is a graph on $v(G)=n$ vertices and
$e(G)=cn^{1+\alpha}$ edges.  We denote the minimum degree in $G$
by $\delta(G)$.  The following lemma shows that we can remove low degree 
vertices from our graph without greatly affecting the number of
edges. 
\begin{lemma}
\label{lem:minDeg}
Let $G$ be as above.  Then $G$ contains a subgraph $H$ with
$e(H)=(c/2)n^{1+\alpha}$ edges such that $\delta(H)\ge (c/2)n^{\alpha}$.
\end{lemma}
\begin{proof}
We remove vertices from $G$ of degree less than $(c/2)n^{\alpha}$.
Then, the resulting subgraph $H$ has $\delta(H)\ge (c/2)n^{\alpha}$ and we
removed fewer than $(c/2)n^{1+\alpha}$ edges so $H$ contains more than $(c/2)n^{1+\alpha}$ edges.
\end{proof}
Note that the subgraph $H$ constructed above contains at least 
$v(H)=\sqrt{c}n^{1/2+\alpha/2}$ vertices.

Suppose we are given a path on $k$ edges $P_k=p_0p_1\dots p_k$.  We call this
path \emph{irredundant} if \[\sum_{i\in I}\overrightarrow{p_ip_{i+1}}\ne 0\] for
every $\emptyset\ne I \subset \{0,1,\dots,k-1\}$.

\begin{proof}[Proof of Theorem~\ref{thm:main}]
Let $G$ be the graph with the $n$ points in the plane as vertices and
the unit distances with $z$ coming from $\Gamma$ as edges.  Suppose there are 
$n^{1+\varepsilon}$ such distances.  Then $e(G)=n^{1+\varepsilon}$.  We will 
show that we can take $\varepsilon$ as small as we like.  We will count the 
number of irredundant paths $P_k$ in $G$, for a fixed $k$ that we will choose 
later.  By Lemma~\ref{lem:minDeg} we can assume that $e(G)\ge 
(1/2)n^{1+\varepsilon}, v(G)\ge n^{1/2+\varepsilon/2}$ and $\delta(G)\ge 
(1/2)n^{\varepsilon}$.

The number of irredundant paths $P_k$ starting at any vertex $v$ is at
least \[N\ge\prod_{\ell=0}^{k-1}(\delta(G)-2^{\ell}+1) \ge 
\frac{n^{k\varepsilon}}{2^{2k}}.\] The first inequality is true since if we have
constructed a subpath $P_\ell$ of $P_k$, then at most $2^{\ell}-1$ of the at 
least $\delta(G)$ possible continuations are forbidden.  In the second 
inequality we have assumed that $2^k \le (1/2)n^{\varepsilon}$, which is true as 
long as $k < \varepsilon\log n/\log 2 - 1$ (we will show that this holds at the 
end of the proof).  Thus the total number of irredundant paths $P_k$ is at least 
$Nn^{1/2+\varepsilon/2}/2\ge n^{1/2+(k+1/2)\varepsilon}/2^{2k+1}$.
It follows that there are two vertices $v$ and $w$ with at least
$Nn^{1/2+\varepsilon/2}/n^2 \ge n^{(k+1/2)\varepsilon-3/2}/4^{k}$
irredundant paths $P_k$ between them.  We will call the set of these paths 
$\mathcal{P}_{vw}$, so that we have \[|\mathcal{P}_{vw}|\ge 
\frac{n^{(k+1/2)\varepsilon-3/2}}{4^{k}}.\]

Given $P_k\in\mathcal{P}_{vw}$, $P_k=p_0p_1\dots p_k$, consider the $k$-tuple
$(z_1, \dots, z_k)$ where $z_i$ is the complex number in the
direction from $p_{i-1}$ to $p_i$, i.e.
$z_i=z(p_{i-1},p_i)=\overrightarrow{p_{i-1}p_i}$.  Let $a=z(v,w)$.  Then 
$z_1+z_2+\dots+z_k=a$.  Since the path is irredundant no subsum on the left 
vanishes.  So $P_k$ corresponds to a nondegenerate solution of 
Equation~\eqref{eq:sum} with $a_i=1/a$ for $i=1,2,\dots,k$.  Thus, by 
Theorem~\ref{thm:subspace}, \[|\mathcal{P}_{vw}|\le(8k)^{4k^4(k+kr+1)}.\]
Putting these inequalities together and taking logarithms we get 
\begin{eqnarray*}
   ((k+1/2)\varepsilon-3/2)\log n &\le& k\log 4 + 4k^4(k+kr+1)\log(8k)\\
                                  &\le& 5rk^5\log k,
\end{eqnarray*}
where the last inequality holds for large $k$.  From this we get
\begin{equation}
   \label{eq:ineq}
   \varepsilon\le \frac{5rk^5\log k}{(k+1/2)\log n}+\frac{3}{2(k+1/2)} \le 
   \frac{5rk^4\log k}{\log n} + \frac{3}{2k}.
\end{equation}

We consider the expression on the right hand side as a function of $k$.  
Optimizing this function we get \[k \ge \exp\biggl((1/5)W(5c_2\log n/r)\biggr)\] 
for some constant $c_2>0$ where $W$ is the positive real-valued function 
satisfying $x=W(x)e^{W(x)}$.  This function is called the Lambert W function and 
was first studied by J.H.~Lambert in 1758 \cite{Lamb58}.  The following 
asymptotic expression is due to N.G.~de~Bruijn \cite{Brui61}:  
\[W(x)=\log(x)-\log\log(x)+O\bigl(\frac{\log\log\log x}{\log\log x}\bigr).\]
We don't require this much accuracy.  One can easily check, and we will just use 
the fact, that $(1/2)\log x\le W(x)\le \log x$ for $x\ge e$.

Then we can take \[c'\biggl(\frac{\log n}{r}\biggr)^{1/5} \le k \le 
c''\biggl(\frac{\log n}{r}\biggr)^{1/5}\] for some constants $c',c''>0$.


For any $\varepsilon > 0$ there is a constant $c>0$ such that if $r+1\le c\log 
n$ then the inequality in \eqref{eq:ineq} holds for large $n$.  When counting 
$P_k$'s we made the assumption that $k\le \varepsilon\log n/\log 2-1$.  Checking 
the above values of $k$ and $f$ we see that this holds for large $n$.

This completes the proof.
\end{proof}

\section{Analysis of \texorpdfstring{Erd\H{o}s'}{Erdos'} lower bound}
\label{sec:lower}

It would be interesting to analyze the possible group structure of unit 
distances from a maximal set of points.  We will now show that the lower bound 
configuration for the unit distance problem given by Erd\H{o}s satisfies the 
hypotheses of Theorem~\ref{thm:main}.  Matou\v{s}ek has given a very in-depth 
account of Erd\H{o}s' lower bound and we will follow that here \cite{Mato02}.

We require the following number theoretic functions: \[\pi_{d,a}(x) = 
   \sum_{\substack{p\le x \\ p\equiv a(d)}} 1,\quad \vartheta_{d,a}(x) = 
   \sum_{\substack{p\le x \\ p\equiv a(d)}} \log p, \quad \psi_{d,a}(x) = 
\sum_{\substack{p^{\ell}\le x \\ p^{\ell}\equiv a(d)}} \log p,\] where the first 
two sums are over primes less than $x$ of the form $p=a+kd$ and the last sum  is 
over primes $p$ and positive integers $\ell$ such that $p^{\ell}=a+kd$ and 
$p^{\ell}\le x$.  These are analogues of the prime counting function and 
Chebyshev functions for arithmetic progressions.

We will use the following results regarding these functions all of which are 
well known in number theory.  For details see \cite{Iwan04} and \cite{Mont07}.
\begin{theorem}[The Prime Number Theorem for Arithmetic Progressions]
   \label{thm:pnt}
   Suppose $a$ and $d$ are positive integers such that $(a,d)=1$.  Then 
   \[\pi_{d,a}(n) = (1+o(1))\frac{1}{\varphi(d)}\cdot\frac{n}{\log n}.\]
\end{theorem}
A simple consequence of this result is that if $a$ and $d$ are positive integers 
such that $a<d$ and $(a,d)=1$ then the $k$th prime of the form $p_i=a+k_id$ 
satisfies $p_k=(1+o(1))k\log k/\varphi(d)$.
\begin{theorem}
   \label{thm:whale}
   Suppose $a$ and $d$ are positive integers such that $(a,d)=1$.  Then 
   \[\psi_{d,a}(n)=(1+o(1))\frac{n}{\varphi(d)}.\]
\end{theorem}
Theorem~\ref{thm:whale} can be deduced from Theorem~\ref{thm:pnt} by partial 
summation.

\begin{theorem}
   \label{thm:theta}
   Suppose $a$ and $d$ are positive integers such that $(a,d)=1$.  Then 
   \[\vartheta_{d,a}(n)=(1+o(1))\psi_{d,a}(n).\]
\end{theorem}
The above two theorems give $\vartheta_{d,a}(n)=(1+o(1))n/\varphi(d)$.

We will also use the following fact.  For details see \cite{Nive91}.
\begin{theorem}
   \label{thm:countTwos}
   The number of integer solutions, $R(m)$, of $x^2+y^2=m$ where $m=p_1p_2\dots 
   p_r$ and the $p_i$ are distinct primes of the form $p_i=4k_i+1$ is \[R(m) = 
   2^{r+2}.\]
\end{theorem}

The lower bound configuration consists of $n$ points in a 
$\sqrt{n}\times\sqrt{n}$ grid.  The step in the grid is chosen to be 
$1/\sqrt{m}$ where $m$ is the product of the first $r-1$ primes of the form 
$4k+1$ and $r$ is the largest number with $m\le n/4$.  We will in fact consider 
a $\sqrt{n}\times\sqrt{n}$ grid with step $1$ and then count the distances of 
length $\sqrt{m}$.  This gives a lower bound to the unit distance problem by 
scaling the point set by $1/\sqrt{m}$.

We have $4p_1p_2\dots p_{r-1}\le n < 4p_1p_2\dots p_r$.  From this the bound 
$r\ge\log n/(3\log\log n)$ is found using the prime number theorem for 
arithmetic progressions.  Distances equal to $\sqrt{m}$ in this configuration 
correspond to integer solutions of $x^2+y^2=m$.  In the lower bound, the fact 
that there are at least $2^{r-1}/16$ such distances is used.  But an upper bound 
on the number of such distances can also be found.  By 
Theorem~\ref{thm:countTwos} there are at most $4.2^{(r-1)+2}=2^{r+3}$ such 
distances from any point so we have at most $2^{r+3}n$ such distances in total.  
Erd\H{o}s' construction gives a lower bound for $r$.  If we can find an upper 
bound for $r$ then we are done as will be described below.

We will briefly explain the reason that $m$ is defined as above as this 
highlights the generators to choose for a multiplicative subgroup of 
$\mathbb{C}^*$ containing the unit distances of the configuration.  A prime $p$ 
has a unique expression, up to the order of the terms, of the form $x^2+y^2=p$ 
with $x$ and $y$ positive integers if and only if $p=2$ or $p=4k+1$ for some 
integer $k$.  The Brahmagupta-Fibonacci identity says that the product of two 
numbers, each expressible as the sum of two squares, is itself expressible as 
the sum of two squares.  Specifically
\begin{eqnarray*}
(a^2+b^2)(c^2+d^2)&=&(ac-bd)^2+(ad+bc)^2\\
                  &=&(ac+bd)^2+(ad-bc)^2.
\end{eqnarray*}
So as we multiply more primes of the form $4k+1$ together we get more 
expressions of the resulting number as a sum of two squares.  So all 
solutions of $x^2+y^2=m$ can be described in terms of the solutions of 
$x_j^2+y_j^2=p_j$.

More formally, we consider the ring $R$ of points in $\mathbb{Z}^2$ with 
addition defined coordinate-wise, so $(a,b)+(c,d)=(a+c,b+d)$, and multiplication 
defined as follows $(a,b)\cdot(c,d)=(ac-bd,ad+bc)$.  One can check that $R$ is 
actually a ring and is in fact isomorphic to the Gaussian integers 
$\mathbb{Z}[i]$ since the operations correspond to complex addition and complex 
multiplication.  But $\mathbb{Z}[i]$ is a unique factorization domain so $R$ is 
also a unique factorization domain.  We will consider the elements of $R$ as 
distance vectors.

In our grid we are looking for the distance $m=p_1\dots p_{r-1}$.  By 
Theorem~\ref{thm:countTwos}, the number of pairs $(x,y)\in\mathbb{Z}^2$ with 
$x^2+y^2=m$ is $R(m)=2^{r+1}$.  Suppose $x_j^2+y_j^2=p_j$.  We consider the 
point $(x_j,y_j)\in R$.  The product of these $r-1$ points $(x,y) = 
(x_1,y_1)(x_2,y_2)\dots(x_{r-1},y_{r-1})$ has magnitude 
\[|(x,y)|=|(x_1,y_1)|\dots|(x_{r-1},y_{r-1})| = \sqrt{p_1\dots p_{r-1}} = 
\sqrt{m}.\]  So this product gives a point with length $\sqrt{m}$.

Now, $R$ is a unique factorization domain.  That means that the point 
$(x,y)=(x_1,y_1)\dots(x_{r-1},y_{r-1})$ has unique factorization.  Specifically, 
in any other factorization of $(x,y)=(x_1',y_1')\dots(x_{r-1}',y_{r-1}')$ there 
is a bijection $\phi$ of the factors such that 
$(x_j,y_j)=u_j(x_{\phi(j)}',y_{\phi(j)}')$ where $u_j$ is a unit.  The units in 
$R$ correspond to the units in $\mathbb{Z}[i]$.  In the latter these are 
$1,-1,i,-i$ so in the former they are $(1,0), (-1,0), (0,1)$ and $(0,-1)$.  Two 
elements $(a,b),(c,d)\in R$ are called \emph{associates} if $(a,b)=u(c,d)$ for 
some unit $u$.  So in a unique factorization domain the factorization of an 
element is unique up to ordering and associates.  Now, since the $p_j$'s are odd 
primes we cannot have $x_j=\pm y_j$ for $1\le j\le r-1$.  One can check that 
$(x_j,y_j)$ and $(x_j,-y_j)$ are not associates and $(x_j,y_j), (x_k,y_k)$ are 
not associates for $j\ne k$.  So we have two points to choose from for each 
$p_j$, namely $(x_j,y_j)$ and $(x_j,-y_j)$, giving $2^{r-1}$ choices for 
$(x,y)$.  None of the factors are associates so these choices for $(x,y)$ are 
all distinct.  If we multiply a given $(x,y)$ by a unit then we get four 
different values.  So we get $4.2^{r-1}=2^{r+1}$ distinct points $(x,y)$ each 
with length $\sqrt{m}$.  So these give all possible required distances by 
Theorem~\ref{thm:countTwos}.  The units are torsion points of $R$ (they have 
finite multiplicative order in $R$) so they don't affect the rank.

Going back to unit distances, if we take the complex numbers 
\[z_j=m^{-1/(2r-2)}(x_j+iy_j),\qquad w_j=m^{-1/(2r-2)}(x_j-iy_j)\] for $1\le 
j\le r-1$ then these generate the multiplicative group of unit distances in the 
configuration.  Thus the unit distances come from a multiplicative subgroup of 
$\mathbb{C}^*$ of rank at most $r-1$.  So we just need to bound $r$ from above.

We do this by looking at the inequality $p_1\dots p_{r-1}\le n/4$.
Taking logarithms we get $\vartheta_{4,1}(p_{r-1})\le \log(n/4)$.  By 
Theorems~\ref{thm:whale} and \ref{thm:theta} we get 
\[\frac{p_{r-1}}{2\sqrt{2}}\le \log(n/4).\]  By the remark after 
Theorem~\ref{thm:pnt} we get 
\[\frac{(r-1)\log(r-1)}{2\sqrt{2}}\le2\sqrt{2}\log(n/4).\]  Solving for $r$ we 
get \[r\le \frac{16\log n}{\log\log n}.\]
Thus the unit distances come from a multiplicative subgroup of rank at most 
$r-1\le 16\log n/\log\log n-1$.  Since \[\frac{16\log n}{\log\log n}-1\le c\log 
n\] for large $n$ this configuration is covered by Theorem~\ref{thm:main}.

\section*{Acknowledgements}
I am very grateful to J\'ozsef Solymosi for the idea of using the subspace 
theorem for this problem and other useful discussions.  I would also like to 
thank Yann Bugeaud for making me aware of Amoroso and Viada's bound for the 
corollary of the subspace theorem giving the improved bound in 
Theorem~\ref{thm:subspace}.  Lastly I would like to thank Christian Elsholtz for 
helpful comments and corrections in the last section of the paper.

\bibliographystyle{plain}
\bibliography{refs2.bib}
\end{document}